\documentclass[11pt]{article}
 \pdfoutput=1
\title{Symplectic embeddings of 4-dimensional ellipsoids into polydiscs}
\author{Max Timmons, Priera Panescu and Madeleine Burkhart}
\date{}

\usepackage{amssymb}
\usepackage{latexsym}
\usepackage{amsmath}
\usepackage{amsthm}
\usepackage{amscd}
\usepackage{verbatim}
\usepackage{enumerate}
\usepackage{graphicx}

\numberwithin{equation}{section}

\newtheorem{theorem}{Theorem}[section]
\newtheorem{proposition}[theorem]{Proposition}

\newtheorem{lemma}[theorem]{Lemma}
\newtheorem{lemma-definition}[theorem]{Lemma-Definition}
\newtheorem{sublemma}[theorem]{Sublemma}

\newtheorem{conjecture}[theorem]{Conjecture}

\theoremstyle{definition}
\newtheorem{definition}[theorem]{Definition}
\newtheorem{remark}[theorem]{Remark}

\newtheorem{example}[theorem]{Example}

\newtheorem*{acknowledgments}{Acknowledgments}

\newcommand{\floor}[1]{\left\lfloor #1 \right\rfloor}
\newcommand{\ceil}[1]{\left\lceil #1 \right\rceil}

\newcommand{\eqdef}{\;{:=}\;}

\newcommand{\C}{{\mathbb C}}
\newcommand{\Q}{{\mathbb Q}}
\newcommand{\R}{{\mathbb R}}
\newcommand{\N}{{\mathbb N}}
\newcommand{\Z}{{\mathbb Z}}

\newcommand{\op}{\operatorname}

\newcommand{\bpm}{\begin{pmatrix}}
\newcommand{\epm}{\end{pmatrix}}

\renewcommand{\epsilon}{\varepsilon}

\begin{document} 
\setcounter{tocdepth}{2}

\maketitle

\begin{abstract}

McDuff and Schlenk have recently determined exactly when a four-dimensional symplectic ellipsoid symplectically embeds into a symplectic ball.  Similarly, Frenkel and M\"uller have recently determined exactly when a symplectic ellipsoid symplectically embeds into a symplectic cube. Symplectic embeddings of more complicated structures, however, remain mostly unexplored.  We study when a symplectic ellipsoid $E(a,b)$ symplectically embeds into a polydisc $P(c,d)$.  We prove that there exists a constant $C$ depending only on $d/c$ (here, $d$ is assumed greater than $c$) such that if $b/a$ is greater than $C$, then the only obstruction to symplectically embedding $E(a,b)$ into $P(c,d)$ is the volume obstruction.  We also conjecture exactly when an ellipsoid embeds into a scaling of $P(1,b)$ for $b$ greater than or equal to $6$, and conjecture about the set of $(a,b)$ such that the only obstruction to embedding $E(1,a)$ into a scaling of $P(1,b)$ is the classical volume.  Finally, we verify our conjecture for $b = \frac{13}{2}$.       

\end{abstract}

\tableofcontents

\section{Introduction}

\subsection{Statement of Results}
\label{sec:results}

Let $(X_0,\omega_0)$ and $(X_1,\omega_1)$ be symplectic manifolds.  A {\em symplectic embedding} of $(X_0,\omega_0)$ into $(X_1,\omega_1)$ is a smooth embedding $\varphi$ such that $\varphi^*(\omega_1)=\omega_0$.  It is interesting to ask  when one symplectic manifold embeds into another.  For example, define the (open) four-dimensional symplectic \textit{ellipsoid}
\begin{equation}
\label{eqn:ellipsoid}
E(a,b) = \left\{(z_1,z_2)\in\C^2\;\bigg|\; \frac{\pi|z_1|^2}{a}+\frac{\pi|z_2|^2}{b}<1\right\},
\end{equation}
and define the (open) \textit{symplectic ball} $B(a) \eqdef E(a,a)$.  These inherit symplectic forms by restricting the standard form $\omega=\sum_{k=1}^2dx_kdy_k$ on $\R^4=\C^2.$  In \cite{MS}, McDuff and Schlenk determined exactly when a four-dimensional symplectic ellipsoid $E(a,b)$ embeds symplectically into a symplectic ball, and found that if $\frac{b}{a}$ is small, then the answer involves an ``infinite staircase" determined by the odd index Fibonacci numbers, while if $\frac{b}{a}$ is large then all obstructions vanish except for the volume obstruction.  

To give another example, define the (open) four-dimensional \textit{polydisc} \begin{equation}
\label{eqn:polydisk}
P(a,b) = \left\{(z_1,z_2)\in\C^2\;\big|\; \pi|z_1|^2< a,\; \pi|z_2|^2< b\right\},
\end{equation}
where $a,b \geq 1$ are real numbers and the symplectic form is again given by restricting the standard symplectic form on $\R^4$.  Frenkel and M\"uller determined in \cite{FM} exactly when a four-dimensional symplectic ellipsoid symplectically embeds into a {\em cube} $C(a) \eqdef P(a,a)$ and found that part of the expression involves the Pell numbers.  Cristofaro-Gardiner and Kleinman \cite{CGK} studied embeddings of four-dimensional ellipsoids into scalings of $E(1,\frac{3}{2})$ and also found that part of the answer involves an infinite staircase determined by a recursive sequence.

Here we study symplectic embeddings of an open four-dimensional symplectic ellipsoid $E(a,b)$ into an open four-dimensional symplectic polydisc $P(c,d)$.  By scaling, we can encode this embedding question as the function \begin{equation}
\label{eqn:inf}
d(a,b) := \op{inf} \lbrace \lambda | E(1,a) \stackrel{s}\hookrightarrow P( \lambda , b \lambda ) \rbrace,
\end{equation}
where, $a$ and $b$ are real numbers that are both greater than or equal to $1$.

The function $d(a,b)$ always has a lower bound, $\sqrt{\dfrac{a}{2b}}$, the volume obstruction. Our first theorem states that for fixed $b$, if $a$ is sufficiently large then this lower bound is sharp, i.e. all embedding obstructions vanish aside from the volume obstruction:

\begin{theorem}
\label{thm: rigidity}
If $a \geq \dfrac{9(b+1)^{2}}{2b}$, then $d(a,b)=\sqrt{\dfrac{a}{2b}}$.
\end{theorem}

This is an analogue of a result of Buse-Hind \cite{BH} concerning symplectic embeddings of one symplectic ellipsoid into another.  

From the previously mentioned work of McDuff-Schlenk, Frenkel-M\"uller, and Cristofaro-Gardiner-Kleinman, one expects that if $a$ is small then the function $d(a,b)$ should be more rich.  Our results suggest that this is indeed the case.  For example, we completely determine the graph of $d(a,\dfrac{13}{2})$ (see Figure 1):

\begin{theorem}
\label{thm:13/2}
For $b=\dfrac{13}{2}$, $d(a,b) \geq \sqrt{\dfrac{a}{13}}$ and is equal to this lower bound for all $a$ except on the following intervals:

(i) $d(a,\dfrac{13}{2})=1$ for all $a\in\left[1,\frac{25}{2}\right]$

(ii) For $0 \leq k \leq 4$, $k\in\Z$:

\begin{center}

$d(a,b)= \left\{
        \begin{array}{ll}
            \dfrac{2a}{25+2k} & \quad a\in[\alpha_{k},13+2k], \\
            \\
            \dfrac{26+4k}{25+2k} & \quad a\in[13+2k,\beta_{k}],
        \end{array}
    \right.$

\end{center}
where $\alpha_{0} = 25/2$, $\alpha_{1} = 351/25$, $\alpha_{2} = 841/52$, $\alpha_{3} =961/52$, $\alpha_{4} = 1089/52$, $\beta_{0}= 351/25$,  $\beta_{1}= 1300/81$, $\beta_{2}= 15028/841$, $\beta_{3}= 18772/961$, and $\beta_{4}= 2548/121$.

\end{theorem}

Interestingly, the graph of $d(a,\frac{13}{2})$ has only finitely many nonsmooth points, in contrast to the infinite staircases in \cite{MS,FM,CGK}.  This appears to be the case for many values of $b$.  For example, we conjecture what the function $d(a,b)$ is for all  $b\geq6,$ see conjecture \ref{conj2}.








\begin{figure}
\includegraphics[height = 75 mm, width = 120 mm]{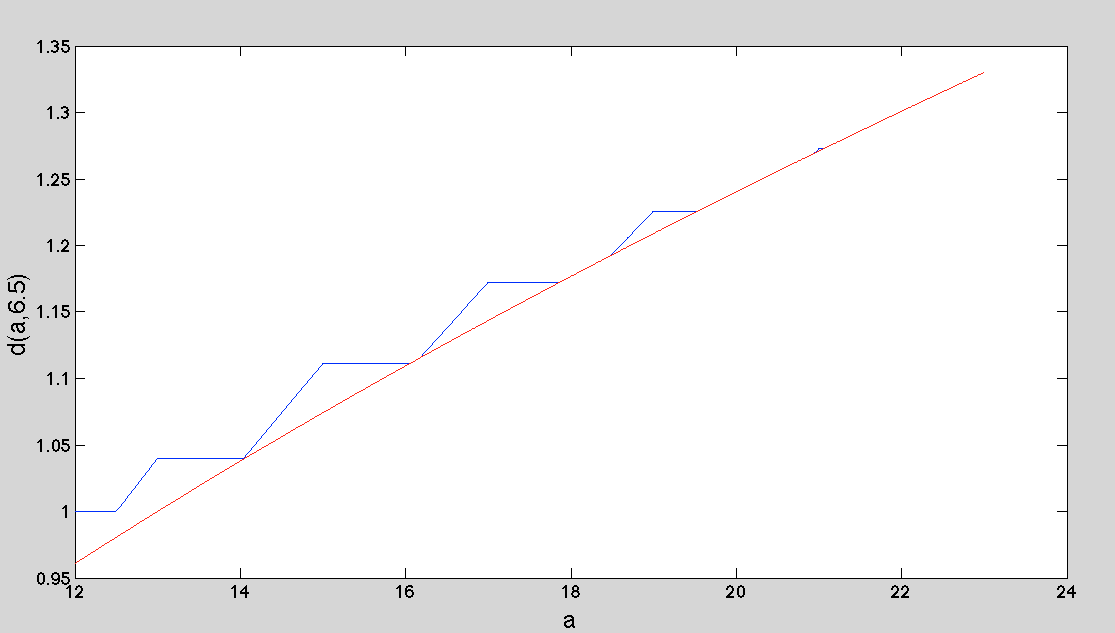}
\caption{The graph of $d(a,\frac{13}{2})$.  The red line represents the volume obstruction.}
\end{figure}


Our proofs rely on the following remarkable theorem of Frenkel and M{\"u}ller \cite{FM}.  Let $N(a,b)$ be the sequence (indexed starting at $0$) of all non-negative integer linear combinations of $a$ and $b$,  arranged with repetitions in non-decreasing order, and let $M(a,b)$ be the sequence whose $k^{th}$ term is 

\begin{center}
$\op{min} \lbrace ma+nb \vert (m+1)(n+1) \geq k+1 \rbrace$
\end{center} 
where $k, m,n \in \Z_{\geq 0}$.  Write  $N(a,b) \leq M(c,d)$ if each term in the sequence $N(a,b)$ is less than or equal to the corresponding term in $M(c,d)$.  Frenkel and M{\"u}ller show that embeddings of an ellipsoid into a polydisc are completely determined by the sequences $M$ and $N$:
\begin{theorem} (Frenkel-M{\"u}ller \cite{FM})
\label{them:(Frenkel-Muller)} 
 There is a symplectic embedding $E(a,b) \stackrel{s}\hookrightarrow P(c,d)$ if and only if $N(a,b) \leq M(c,d)$.
\end{theorem}

To motivate the sequences $M$ and $N$, note that $N$ is the sequence of {\em ECH capacities} of the symplectic ellipsoid $E(a,b)$ while $M$ is the sequence of ECH capacities of the symplectic polydisc $P(c,d)$.  The ECH capacities are a sequence of nonnegative (possibly infinite) real numbers, defined for any symplectic four-manifold, that obstruct symplectic embeddings.  We will not discuss ECH capacities here; see \cite{H} for a survey.  Theorem~\ref{them:(Frenkel-Muller)} is equivalent to the statement that the ECH capacities give sharp obstructions to embeddings of an ellipsoid into a polydisc.  



\begin{acknowledgments}
We wish to thank Daniel Cristofaro-Gardiner for his helpful explanations, reference suggestions, encouragement and patience. We also thank the NSF, Michael Hutchings and UC Berkeley for providing the opportunity to work on symplectic embedding problems this summer.
\end{acknowledgments}

\section{Proof of Theorem \ref{thm: rigidity}}

\subsection{Weight sequences and the \# operation}
We begin by describing the machinery that will be used to prove Theorem \ref{thm: rigidity}.

Let $a^2$ be a rational number.  In \cite{Mc}, McDuff shows that there is a finite sequence of numbers 
\[W(1, a^{2})= (a_{1},...,a_{n}),\] 
called the  \textit{(normalized) weight sequence for $a^{2}$}, such that $E(1,a^{2})$ embeds into a symplectic ellipsoid if and only if the disjoint union $\sqcup B(W) := \sqcup B(a_{i})$ embeds into that ellipsoid. 

To describe the weight sequence, let 
\begin{equation}
\label{equ: weightseq}
W(a^{2},1)= (X_{0}^{\times \ell_{0}}, X_{1}^{\times \ell_{1}},...,X_{k}^{\times \ell_{k}})
\end{equation}
where $X_{0} > X_{1} >...> X_{k}$ and $ \ell_{k} \geq 2$. The $\ell_{i}$ are the multiplicities of the entries $X_{i}$ and come from the continued fraction expansion
\begin{center} 
$a^{2} = \ell_{0} + \dfrac{1}{\ell_{1} + \dfrac{1}{\ell_{2}+...\dfrac{1}{\ell_{k}}}} := [\ell_{0}; \ell_{1},...,\ell_{k}]$.
\end{center}
The entries of \ref{equ: weightseq} are defined as follows:
\begin{center}
$X_{-1}:=a^{2}, X_{0}=1, X_{i+1}=X_{i-1}-\ell_{i}X_{i}, i \geq 0$.
\end{center}
Important results of the weight sequence include that
\begin{equation}
\label{equ: ai}
\Sigma_{i}a_{i}^{2}=a^{2}
\end{equation} and
\begin{equation}
\Sigma_{i}a_{i}=a^2+1-\dfrac{1}{q}
\end{equation}
where for all $i$, $a_{i} \leq 1$ and $a=\dfrac{p}{q}$.

We will also make use of a helpful operation, $\#$, as in \cite{Mc}. Suppose $s_{1}$ and $s_{2}$ are sequences indexed with $k \in \mathbb{Z}$, starting at 0. Then,
\begin{center}
$(s_{1} \# s_{2})_k = sup_{i+j=k}(s_{1})_i+(s_{2})_j$.
\end{center}
A useful application of $\#$ is the following lemma:
\begin{lemma} (McDuff \cite{Mc})
For all $a,b > 0$, we have $N(a,a) \# N(a,b) = N(a,a+b)$. More generally, for all $\ell \geq 1$, we have $(\#^{\ell}N(a,a)) \# N(a,b) = N(a,b+ \ell a)$.
\end{lemma}This lemma together with the weight sequence and scaling implies that 
\begin{equation}
\label{equ: pound}
N(1,a^{2})= N(a_{1},a_{1})\#...\#N(a_{n},a_{n}). 
\end{equation}
Similar to McDuff \cite{Mc}, this machinery allows us to reduce Theorem \ref{thm: rigidity} to a ball-packing problem.

\subsection{Proof of Theorem \ref{thm: rigidity}}
We begin by noting that the ECH capacities for $B(a)$ are
\begin{center}
$N(a,a) = (0, a, a, 2a, 2a, 2a, 3a, 3a, 3a,...)$.
\end{center}
where the terms $N_{k}(a,a)$ of this sequence are of the form $da$ and for each $d$ there are $d+l$ entries occurring at 

\begin{equation}
\label{equ: N terms}
\dfrac{1}{2}(d^{2}+d) \leq k \leq \dfrac{1}{2}(d^{2}+3d).
\end{equation}

Similarly, for the sequence $\dfrac{a}{\sqrt{2b}}M(1,b)$, each term $\dfrac{a}{\sqrt{2b}}M_k(1,b)$ is of the form $d \dfrac{a}{\sqrt{2b}}$ where 

\begin{equation}
\label{equ: M terms}
k \leq \dfrac{d^{2}}{4b}+\dfrac{(1+b)d}{2b}+\dfrac{b^{2}-2b+1}{4b}.
\end{equation}

By scaling and continuity, we can study $d(a^{2},b)$ with $a^2$ rational. So, we can prove that the volume obstruction is the only obstruction when $a \geq \dfrac{3(b+1)}{\sqrt{2b}}$ by showing that 
\begin{equation}
N(1,a^{2}) \leq \dfrac{a}{\sqrt{2b}} M(1,b)
\end{equation}
for said $a$ values.

By \ref{equ: N terms} and \ref{equ: M terms}, it is therefore sufficient to show that 
\begin{equation}\label{eq1}
\Sigma_{i}d_{i}a_{i} \leq \dfrac{a}{\sqrt{2b}}d
\end{equation}
whenever $d_{1},..,d_{m},d$ are nonnegative integers such that
\begin{equation}\label{eq2}
\Sigma_{i}(d_{i}^{2}+d_{i}) \leq 2(\dfrac{d^{2}}{4b}+\dfrac{(1+b)d}{2b}+\dfrac{b^{2}-2b+1}{4b}).
\end{equation}

We do so by considering the following cases:

\textit{Case 1.} $\Sigma_{i}(d_{i}^{2}) \leq \dfrac{d^{2}}{2b}$. In this case, the Cauchy-Schwartz Inequality along with \ref{equ: ai} implies \ref{eq1}.

\textit{Case 2.} $\Sigma_{i}(d_{i}^{2}) > \dfrac{d^{2}}{2b}$. This case along with \ref{eq2} implies
\begin{center}
$\Sigma_{i}d_{i}a_{i} \leq \Sigma_{i}d_{i} \leq \dfrac{(1+b)d}{b}+\dfrac{b^{2}-2b+1}{2b}$.
\end{center}
So, we need
\begin{center}
$\dfrac{(1+b)d}{b}+\dfrac{b^{2}-2b+1}{2b} \leq \dfrac{a}{\sqrt{2b}}d$.
\end{center}
It follows that
\begin{equation}
\label{equ: initialbound}
a \geq \dfrac{b+1}{\sqrt{2b}}(2+\dfrac{b+1}{d}).
\end{equation}

Now let $d=b+1$. We see that \ref{equ: M terms} is equivalent to

\begin{center}
$k \leq b+1+\dfrac{1}{4b}.$
\end{center}

It is easy to see that $N_{k}(1,a^{2}) \leq \dfrac{a}{\sqrt{2b}} M_{k}(1,b)$ for all such $k$ values. As such, we can apply $d=b+1$ to \ref{equ: initialbound} to get

\begin{equation}
\label{equ: result}
a \geq \dfrac{3(b+1)}{\sqrt{2b}},
\end{equation}
hence the desired result. $\qed$

\begin{remark}\label{sharpbound}
We allow $d=b+1$ in the statement of Theorem 1.4. However, if we show $N_{k}(1,a^{2}) \leq \dfrac{a}{\sqrt{2b}} M_{k}(1,b)$ for all $k \leq \dfrac{d^{2}}{4b}+\dfrac{(1+b)d}{2b}+\dfrac{b^{2}-2b+1}{4b}$, then we can use this $d$ in \ref{equ: initialbound} to achieve a sharper bound for a.
\end{remark}

\section{Proof of Theorem \ref{thm:13/2} Part I}
\label{sec:parti}

We begin by computing $d(a,\frac{13}{2})$ on the regions where it is linear.

\subsection{Nondifferentiable points and Ehrhart polynomials}
\label{sec:critical}

We first compute the values of $d$ at certain points.  These will eventually be the points $a$ where $d(a,\frac{13}{2})$ is not differentiable.

\begin{proposition}
\label{prop:criticalvalues}
We have:

\begin{center}
$\begin{array}{lll}
d\left(1,\dfrac{13}{2}\right)=1, & d\left(\dfrac{25}{2},\dfrac{13}{2}\right)=1, & d\left(13,\dfrac{13}{2}\right)=\dfrac{26}{25}, \\
d\left(\dfrac{351}{25},\dfrac{13}{2}\right)=\dfrac{26}{25}, & d\left(15,\dfrac{13}{2}\right)=\dfrac{10}{9}, & d\left(\dfrac{1300}{81},\dfrac{13}{2}\right)=\dfrac{10}{9}, \\
d\left(\dfrac{841}{52},\dfrac{13}{2}\right)=\dfrac{29}{26}, & d\left(17,\dfrac{13}{2}\right)=\dfrac{34}{29}, & d\left(\dfrac{15028}{841}\right)=\dfrac{34}{29}, \\
d\left(\dfrac{961}{52},\dfrac{13}{2}\right)=\dfrac{31}{26}, & d\left(19,\dfrac{13}{2}\right)=\dfrac{38}{31}, & d\left(\dfrac{18772}{961},\dfrac{13}{2}\right)=\dfrac{38}{31}, \\
d\left(\dfrac{1089}{52},\dfrac{13}{2}\right)=\dfrac{33}{26}, & d\left(21,\dfrac{13}{2}\right)=\dfrac{42}{33},\text{     and} & d\left(\dfrac{2548}{121},\dfrac{13}{2}\right)=\dfrac{42}{33}.
\end{array}$
\end{center}
\end{proposition}

To prove the proposition, the main difficulty comes from the fact that that applying Theorem~\ref{them:(Frenkel-Muller)} in principle requires checking infinitely many ECH capacities.  Our strategy for overcoming this difficulty is to study the growth rate of the terms in the sequences $M$ and $N$.  We will find that in every case needed to prove Proposition~\ref{prop:criticalvalues}, one can bound these growth rates to conclude that only finitely many terms in the sequences need to be checked.  This is then easily done by computer.  The details are as follows:  

\begin{proof}
{\em Step 1.}  For the sequence $N(a,b)$, let $k(a,b,t)$ be the largest $k$ such that $N_{k}(a,b)\le t$.   Similarly, for the sequence $M(c,d)$, let  $l(c,d,t)$ be the largest $l$ such that $M_{l}(c,d)\leq t$.  To show that $E(a,b) \stackrel{s} \hookrightarrow P(c,d)$, by Theorem~\ref{them:(Frenkel-Muller)}, we just have to show that for all $t$, we have $k(a,b,t) \geq l(c,d,t)$.  

{\em Step 2.}  We can estimate $k(a,b,t)$ by applying the following proposition:  

\begin{proposition}
\label{thm:ehrhart}
If $a, b, r,$ and $t$ are all positive integers, then $k(\frac{a}{r},\frac{b}{r},t)=$
\begin{equation}
\label{eqn:ehrhartpolynomial}
\begin{aligned}
\frac{1}{2ab}(tr)^{2}+\frac{1}{2}(tr)\left(\frac{1}{a}+\frac{1}{b}+\frac{1}{ab}\right)+\frac{1}{4}\left(1+\frac{1}{a}+\frac{1}{b}\right)+\frac{1}{12}\left(\frac{a}{b}+\frac{b}{a}+ \frac{1}{ab}\right)\\
 +\frac{1}{a}\sum_{j=1}^{a-1}\dfrac{\xi_{a}^{j(-tr)}}{(1-\xi_{a}^{jb})(1-\xi_{a}^{j})}+\frac{1}{b}\sum_{l=1}^{b-1}\dfrac{\xi_{b}^{l(-tr)}}{(1-\xi_{b}^{la})(1-\xi_{b}^{l})}, 
\end{aligned}
\end{equation}
where $\xi_{d}=e^{\frac{2\pi i}{d}}$.
\end{proposition}


\begin{proof}
The number of terms in $N(\frac{a}{r},\frac{b}{r})$ that are less than $t$ is the same as the number of lattice points $(m,n)$ in the triangle bounded by the positive $x$ and $y$ axes and the line $x\frac{a}{r}+y\frac{b}{r} \le t$.  For integral $t$, this number can be computed by applying the theory of ``Ehrhart polynomials". Proposition~\ref{thm:ehrhart} follows by applying \cite[Thm. 2.10]{BR}.       
\end{proof}

We will be most interested in this proposition in the case where $a=r$.  Note that by the last two terms of the formula in Proposition~\ref{thm:ehrhart}, $k(\frac{a}{r},\frac{b}{r},t)$ is a periodic polynomial with period $ab$.  

We also need an argument to account for the fact that Proposition~\ref{thm:ehrhart} is only for integer $t$, whereas the argument in step $1$ involves real $t$.  To account for this we use an asymptotic argument.  Specifically, for $E(1,\frac{a}{r})$ , $a,r\in\Z_{\geq1},$ we bound the right hand side of \eqref{eqn:ehrhartpolynomial} from below by taking the floor function of $t$.  It is convenient for our argument to further bound this expression from below by
\begin{equation}
\label{eqn:finalbound}
\dfrac{c_{1}}{r^{2}}(rt-1)^{2}+\dfrac{c_{2}}{r}(rt-1)+c_{3}.
\end{equation}
where the $c_i$ are the coefficients of the right hand side of \eqref{eqn:ehrhartpolynomial} that do not involve $t$ or $r$.

This is the lower bound that we will use for $k(1,\frac{a}{r},t)$.  




{\em Step 3.}  To get an upper bound $l(c,d,t)$ for $M(c,d)$, recall that $M_{l}(c,d)=min\{cm+dn:(m+1)(n+1)\geq l+1\}$.  For $cm+dn=t$, we solve for $m$ in terms of $n$ and find:
\begin{center}
$\left(\dfrac{t-dn}{c}+1\right)(n+1)-1\geq l$.
\end{center}
Considering $m,n\in\R$, we can take the derivative of the left side of the inequality with respect to $n$ and then set the expression equal to 0 to maximize it.  We do the same with $m$ to obtain:
\begin{center}
$\left(\dfrac{t}{2d}+\dfrac{c}{2d}+\dfrac{1}{2}\right)\left(\dfrac{t}{2c}+\dfrac{d}{2c}+\dfrac{1}{2}\right)-1\geq l$.
\end{center}
By simplifying, we get that an upper bound for $l$ is:
\begin{equation}
\label{eqn:lupper}
l(c,d,t)=\dfrac{t^{2}}{4cd}+\dfrac{(c+d)t}{2cd}+\dfrac{(c-d)^{2}}{4cd}.
\end{equation}

Our strategy now is to check for each point in Proposition~\ref{prop:criticalvalues} that we have  $k(a,b,t)\geq l(c,d,t)$ asymptotically in $t$ for the corresponding $(a,b,c,d)$.  From there, we can check that for a sufficient number of terms, $N(1,a)\leq M(\lambda,\lambda b)$.

{\em Step 4.}  Since the rest of the proof amounts to computation, it is best summarized by the chart below.  In the chart, $k_{t^2}$ and $l_{t^2}$ denote the coefficients of the quadratic terms in the upper and lower bounds from steps $2$ and $3$, while $k_t$ and $l_t$ denote the corresponding coefficients of the linear terms.  

The $t$ column gives a sufficient number to check up to before the asymptotic bounds from the previous three steps are enough.  Note that if the $t^2$ coefficients in any row are equal, then linear coefficients are used to make an asymptotic argument; this explains the appearance of the ``N/A"s in the table.  It is simple to check by computer that the relevant $N$ and $M$ sequences in each row satisfy $N \le M$ once one knows that the problem only has to be checked up to the $t$ in the $t$ column.  Code for this is contained in  \ref{app:embedcode}.

The ECH obstruction column gives an ECH capacity that shows that one cannot shrink $\lambda$ further, i.e. the claimed embeddings are actually sharp.



\begin{center}
\begin{tabular}{|c|c|c|c|c|c|c|}
\hline
$E(1,a)\stackrel{s}\hookrightarrow P(\lambda,\lambda b)$ & $k_{t^{2}}$ & $l_{t^{2}}$ & $k_{t}$ & $l_{t}$ & $t$ & ECH obstruction \\
\hline
$E(1,\frac{25}{2})\stackrel{s}\hookrightarrow P(1,\frac{13}{2})$ & $\frac{1}{25}$ & $\frac{1}{26}$ & N/A & N/A & 51 & 1\\
\hline
$E(1,13)\stackrel{s}\hookrightarrow P(\frac{26}{25},\frac{169}{25})$ & $\frac{1}{26}$ & $\frac{625}{17576}$ & N/A & N/A & 33 & 13 \\
\hline
$E(1,\frac{351}{25})\stackrel{s}\hookrightarrow P(\frac{26}{25},\frac{169}{25})$ & $\frac{25}{702}$ & $\frac{625}{17576}$ & N/A & N/A & 522 & 13 \\
\hline
$E(1,15)\stackrel{s}\hookrightarrow P(\frac{10}{9},\frac{65}{9})$ & $\frac{1}{30}$ & $\frac{81}{2600}$ & N/A & N/A & 29 & 15 \\
\hline
$E(1,\frac{1300}{81})\stackrel{s}\hookrightarrow P(\frac{10}{9},\frac{65}{9})$ & $\frac{81}{2600}$ & $\frac{81}{2600}$ & $\frac{691}{1300}$ & $\frac{27}{52}$ & 272 & 15 \\
\hline
$E(1,\frac{841}{52})\stackrel{s}\hookrightarrow P(\frac{29}{26},\frac{29}{4})$ & $\frac{26}{841}$ & $\frac{26}{841}$ & $\frac{447}{841}$ & $\frac{15}{29}$ & 122 & 17 \\
\hline
$E(1,17)\stackrel{s}\hookrightarrow P(\frac{34}{29},\frac{221}{29})$ & $\frac{1}{34}$ & $\frac{841}{30056}$ & N/A & N/A  & 27 & 17 \\
\hline
$E(1,\frac{15028}{841})\stackrel{s}\hookrightarrow P(\frac{34}{29},\frac{221}{29})$ & $\frac{841}{30056}$ & $\frac{841}{30056}$ & $\frac{7935}{15028}$ & $\frac{435}{884}$ & 32 & 17 \\
\hline
$E(1,\frac{961}{52})\stackrel{s}\hookrightarrow P(\frac{31}{26},\frac{31}{4})$ & $\frac{26}{961}$ & $\frac{26}{961}$ & $\frac{507}{961}$ & $\frac{15}{31}$ & 23 & 19 \\
\hline
$E(1,19)\stackrel{s}\hookrightarrow P(\frac{38}{31},\frac{247}{31})$ & $\frac{1}{38}$ & $\frac{961}{37544}$ & N/A & N/A & 7 & 19 \\
\hline
$E(1,\frac{18772}{961})\stackrel{s}\hookrightarrow P(\frac{38}{31},\frac{247}{31})$ & $\frac{961}{37544}$ & $\frac{961}{37544}$ & $\frac{759}{1444}$ & $\frac{465}{988}$ & 28 & 19 \\
\hline
$E(1,\frac{1089}{52})\stackrel{s}\hookrightarrow P(\frac{33}{26},\frac{33}{4})$ & $\frac{26}{1089}$ & $\frac{26}{1089}$ & $\frac{571}{1089}$ & $\frac{15}{33}$ & 14 & 21 \\
\hline
$E(1,21)\stackrel{s}\hookrightarrow P(\frac{42}{33},\frac{273}{33})$ & $\frac{1}{42}$ & $\frac{121}{5096}$ & N/A & N/A & 26 & 21 \\
\hline
$E(1,\frac{2548}{121})\stackrel{s}\hookrightarrow P(\frac{42}{33},\frac{273}{33})$ & $\frac{121}{5096}$ & $\frac{121}{5096}$ & $\frac{1335}{2548}$ & $\frac{165}{364}$ & 41 & 21 \\
\hline
\end{tabular}
\end{center}
\begin{center}
Table 3.1
\end{center}
\end{proof}

\subsection{The linear steps}
\label{sec:linear steps}

Given the computations from the previous section, the computation of $d(a,\frac{13}{2})$ for all the ``linear steps", i.e. those portions of the graph of $d$ for which $d$ is linear, is straightforward.  Indeed, we have the following two lemmas:

\begin{lemma}
\label{lem:monotonic}
For a fixed $b$, $d(a,b)$ is monotonically non-decreasing.
\end{lemma}
\begin{proof}
This follows from the fact that $E(1,a)\stackrel{s}\hookrightarrow E(1,a^{'})$ if  $a\leq a^{'}$.
\end{proof}


\begin{lemma}
\label{lem:subscaling}
$d(\lambda a,b)\leq \lambda d(a,b)$ (subscaling)
\end{lemma}
\begin{proof}
This follows from the fact that $E(1,\lambda a)\stackrel{s}\hookrightarrow E(\lambda,\lambda a)$ for $\lambda\geq1$.
\end{proof}






By monotonicity, we know that $d(a,\frac{13}{2})$ is constant on the intervals:
\[a\in\left[1,\dfrac{25}{2}\right], \left[13,\dfrac{351}{25}\right], \left[15,\dfrac{1300}{81}\right], \left[17,\dfrac{15028}{841}\right], \left[19,\dfrac{18772}{961}\right], \left[21,\dfrac{2548}{121}\right].\]

We now explain why for $0\leq k\leq4, k\in\Z$,
\begin{center}
$d(a,\frac{13}{2})=\dfrac{2a}{25+2k}\quad a\in[\alpha_{k},13+2k]$,
\end{center}
where $\alpha_{0}=\dfrac{25}{2}, \alpha_{1}=\dfrac{351}{25}, \alpha_{2}=\dfrac{841}{52}, \alpha_{3}=\dfrac{961}{52},$ and $\alpha_{4}=\dfrac{1089}{52}$.

Given the critical points we have determined, along with the subscaling lemma, we have $\dfrac{2a}{25+2k}$ as an upper bound for $d(a,\frac{13}{2})$ on the above intervals.

\subsection{Intervals on which $d(a,\frac{13}{2})$ is linear}
\label{sec:linear}

We also know that:
\begin{center}
$d(a,\frac{13}{2})=sup\left\{\dfrac{N_{x}(1,a)}{M_{x}(1,\frac{13}{2})}:x\in\N\right\}\geq\dfrac{N_{l}(1,a)}{M_{l}(1,\frac{13}{2})}$ for any $l$.
\end{center}

Here is a representative example of our method:

\begin{example}To illustrate how this can give us a suitable lower bound, consider the case where $x=13$:
\begin{center}
$sup\left\{\dfrac{N_{x}(1,a)}{M_{x}(1,\frac{13}{2})}:x\in\N\right\}\geq\dfrac{N_{13}(1,a)}{M_{13}(1,\frac{13}{2})}=\dfrac{2a}{25}$
\end{center}

for $a\in\left[\frac{25}{2},13\right]$.

This lower bound equals the upper bound given by Lemma \ref{lem:subscaling}, so we have proven Theorem \ref{thm:13/2} for $a\in[\frac{25}{2},13]$.
\end{example}

The general method is similar: given $a\in[\alpha_{k},13+2k]$, we can find an $l$ such that:
\begin{center}
$\dfrac{N_{l}(1,a)}{M_{l}(1,\frac{13}{2})}=\dfrac{2a}{25+2k}.$
\end{center}

Such obstructing values of $l$ are given in the following table:

\begin{center}
\begin{tabular}{|c|c|c|}
\hline
$k$ & $\dfrac{2}{25+2k}$ & $l$ \\
\hline
0 & $\frac{2}{25}$ & 13 \\
\hline
1 & $\frac{2}{27}$ & 15 \\
\hline
2 & $\frac{2}{29}$ & 17 \\
\hline
3 & $\frac{2}{31}$ & 19 \\
\hline
4 & $\frac{2}{33}$ & 21 \\
\hline
\end{tabular}
\end{center}
\begin{center}
Table 3.2
\end{center}

Given $a\in[\alpha_{k},13+2k]$ for each integer $k\in[0,4]$, we have found that the upper and lower bounds of $d(a,\frac{13}{2})$ equal $\dfrac{2a}{25+2k}$. Thus, we have proven our claim for these intervals.

\section{Proof of Theorem~\ref{thm:13/2} Part II}

To complete the proof of Theorem~\ref{thm:13/2}, we need to show that aside from the linear steps described in the previous section, the graph of $d(a,\frac{13}{2})$ is equal to the graph of the volume obstruction.  To do this, we adapt some of the ideas from \cite{MS} in a purely combinatorial way.  This will be needed to complete the proof of Theorem~\ref{thm:13/2}.  Our combinatorial perspective on the techniques from \cite{MS} borrows many ideas from \cite{Mc}.

\subsection{Preliminaries}

This section collects the main combinatorial machinery that will be used to complete the proof.  The basic idea behind our proof will be to reduce to a ball packing problem, as in the proof of Theorem~\ref{thm: rigidity}.  The machinery we develop here will be useful for approaching this ball packing problem.   

We begin with two definitions:

\begin{definition} Let $Cr(d,d_i)=(d',d_i')$ where $d'=2d-d_1-d_2-d_3, d_i'=d-d_j-d_k$ for ${i,j,k}={1,2,3}$ and $d_i'=d_i\text{ }$ for all $i\geq4.$ We say $Cr$ is the {\em Cremona transform}.
\end{definition}

\begin{definition} We say $(d,d_i) \in \R^{1+n}$ is:

(i) $positive$ if $d,d_i\geq0$ for all $i$,

(ii) $ordered$ if $d_i,d_{i+1}\neq0$ implies $d_i\geq d_{i+1}$ and $d_i\neq0,d_j=0$ implies $i<j$,

(iii) $reduced$ if positive, ordered, and $d\geq d_1+d_2+d_3$.
\end{definition}

\begin{remark} It will be important to note that $Cr(Cr(d,d_i))=(d,d_i)$.
\end{remark}

We now define a product analogous to the intersection product in \cite{MS}:

\begin{definition} $(x,x_i)\cdot(y,y_i)=xy-\sum_ix_iy_i.$\end{definition}

We also define a vector $-K \in \R^{1+n}$ that is motivated by the the standard anti-canonical divisor in the $M$-fold blow up of $\C P^2$.
\begin{definition} $-K=(3,1,1,\dots,1)$\end{definition}

The following is a combinatorial analogue of ``intersection positivity" that will be useful:

\begin{lemma}\label{1} If $(x,x_i)$ is reduced, $(d,d_i)$ is positive, $-K\cdot(d,d_i)\geq0$, and $d\geq max(d_i),$ then $(x,x_i)\cdot(d,d_i)\geq0$.\end{lemma}
\begin{proof} Let $(d',d_i')$ be ordered $(d,d_i)$. As 
\[(x,x_i)\cdot (d,d_i)\geq (x,x_i)\cdot (d',d_i')\] we can assume without loss of generality that $(d,d_i)$ is ordered. If $x_3=0$ then $x_i=0$ for $i>3$ and \[(x,x_i)\cdot(d,d_i)=xd-x_1d_1-x_2d_2\]
  as $d\geq max(d_i)$.  We know that this expression is greater than or equal to \[(x-x_1-x_2)d\]
as $(x,x_i)$ is reduced, and this is greater than or equal to $0$.  

We now assume without loss of generality that $x_3=1$. Hence, $x_i\leq1$ for $i\geq3$. Let $e_1=x_1-1, e_2=x_2-1.$ Then
\[xd\geq(3+e_1+e_2)d\]
as $(x,x_i)$ is reduced.  This expression is equal to 
\[3d+de_1+de_2\]
as $d\geq d_1,d_2$.  We now have the following chain of inequalities: 
\[3d+de_1+de_2 \geq3d+d_1e_1+d_2e_2\]
\[\geq\sum_id_i+d_1e_1+d_2e_2\]
\[=d_1x+d_2x+\sum_{i\geq3}d_i\]
\[\geq d_1x_1+d_2x_2+\sum_{i\geq3}x_id_i\]
\[=\sum_id_ix_i.\]\end{proof}

In \cite{MS}, Cremona transformations preserve the intersection product. Here we prove an analogous result. 

\begin{lemma}\label{2}$Cr(x,x_i)\cdot Cr(y,y_i)=(x,x_i)\cdot(y,y_i).$\end{lemma}
\begin{proof} $Cr(x,x_i)\cdot Cr(y,y_i)=x'y'-\sum_ix_i'y_i'$
\newline$=(2x-x_1-x_2-x_3)(2y-y_1-y_2-y_3)-(x-x_2-x_3)(y-y_2-y_3)-(x-x_1-x_3)(y-y_1-y_3)-(x-x_2-x_3)(y-y_2-y_3)-\sum_{i>3}x_iy_i$
\[=xy-x_1y_1-x_2y_2-x_3y_3-\sum_{i>3}x_iy_i\]
\[=xy-\sum_{i}x_iy_i=(x,x_i)\cdot(y,y_i).\]
\end{proof}

The following three sets will also be useful:

\begin{definition}$F=\{(d,d_i) | (d,d_i)\cdot(-K+(d,d_i))\geq0, \quad d,d_i\in \Z\}. $\end{definition}
\begin{definition}$F^{+}=\{(d,d_i) | (d,d_i)\in F, \quad d,d_i\geq0\}.$\end{definition}
\begin{definition}$E=\{(d,d_i) | (d,d_i)\cdot(d,d_i)\geq-1, -K\cdot(d,d_i)=1, \quad d,d_i \in \Z\}.$\end{definition}


Also observe: 

\begin{remark}$Cr(F)\subset F$ and $Cr(E)\subset E$. Additionally, $F,F^{+}, \text{and }  E$ are invariant under permutations of $d_i$.\end{remark}
\begin{remark}$(0,-1,0,\cdots,0)\in E.$\end{remark}

Moreover, also define:

\begin{definition}\label{C} Let $C$ be the set of $(x,x_i)$ such that $x,x_i \in \Z$ and \begin{enumerate}[a)]
\item$ (x,x_i)\cdot(x,x_i)\geq0$
\item$(x,x_i)\cdot(d,d_i)\geq0 \quad \forall (d,d_i)\in E.$\end{enumerate}\end{definition}

Both Li-Li \cite{LL} and Mcduff-Schlenk \cite{MS} have found that compositions of Cremona transformations and permutations can reduce certain classes. Here we prove a combinatorial version of those lemmas.
\begin{lemma}\label{3} If $(x,x_i)\in C$ then by a sequence of Cremona transforms and permutations of $x_i's$ we can transform $(x,x_i)$ to $(x',x_i')$ where $(x',x_i')$ is reduced.\end{lemma}
\begin{proof} We begin with some helpful results:
\begin{sublemma}\label{.1} $Cr(C)\subset C.$\end{sublemma}
\begin{proof}The fact that Cr preserves a) follows from the fact that Cr preserves the intersection product.  To complete the sublemma, note that if $(d,d_i)\in E$, then $\quad$ \[Cr(x,x_i)\cdot(d,d_i)=Cr^2(x,x_i)\cdot(d',d_i')=(x,x_i)\cdot(d',d_i')\geq0 \text{ as } (d',d_i')\in E.\]\end{proof}
\begin{sublemma}\label{.2} If $P$ is some permutation, $P(C)\subset C.$\end{sublemma}
\begin{proof} If $(d,d_i)\in E$, then \[P(x,x_i)\cdot(d,d_i)=(x,x_i)\cdot P^{-1}(d,d_i) \text{ as }P^{-1}(E)\subset E.\]\end{proof}
\begin{sublemma}\label{.3} If $(x,x_i) \in C$, then $x,x_i\geq0.$\end{sublemma}
\begin{proof} If $d_i=-\delta_{ij}, (0,d_i)\in E\text{ }$ then $j\leq \text{length}(d_i)$ for all $j$. So, $(x,x_i)\cdot(0,d_i)=x_j\geq0.$ We also have $(x,x_i)\cdot(1,1,1,0,0,\cdots,0)=x-x_1-x_2\geq0.$ As $x_1,x_2\geq0,$ this implies that $x\geq0.$\end{proof}
 Let $oCr$ denote the transformation $Cr$ followed by ordering the $d_i's$. Fix $(x,x_i)\in C.$ Let $(x^k,x_i^k)=oCr^k(x,x_i).$ Let $\alpha(k)=x^k-x_1^k-x_2^k-x_3^k.$ It suffices to show $\alpha(k)\geq0$ for some $k.$ Assume not. Then $\alpha(k)\leq-1$ for all $k$. By Sublemmas \ref{.1} and \ref{.2}, $oCr(C) \subset C$. For $k\geq1,$ \[ x^k=x^{k-1}+\alpha(k-1)\leq x^{k-1}+-1.\] Thus, there exists $k$ such that $x^k<0$. This contradicts Sublemma \ref{.3} completing the proof that we may reduce $(x,x_i)$
\end{proof}

We now prove a result analogous to \cite[Proposition 1.2.12(i)]{MS}.
\begin{lemma}\label{4} If $(x,x_i)\in C \text{ then } (x,x_i)\cdot(d,d_i)\geq0 \text{ for all } (d,d_i)\in F.$\end{lemma}
\begin{proof} By Lemma \ref{3} there exists $A$, a composition of $Cr$ and permutations, such that $A(x,x_i)=(x',x_i') \text{ with } (x',x_i')$ reduced. For $(d,d_i) \in F, \text{ let } A(d,d_i)=(d',d_i')\in F$. So, \[(x,x_i)\cdot(d,d_i)=A(x,x_i)\cdot A(d,d_i)=(x',x_i')\cdot(d',d_i').\]
Let $e=d, e_i=d_i \text{ if } d_i>0$ and $e_i=0 \text{ if } d_i'<0$. We note $(e,e_i)\in F \text{ and }$ \[(x',x_i')\cdot(d',d_i')\geq(x',x_i')\cdot(e,e_i).\] If $(e,e_i)\cdot(e,e_i)\geq0$ then Cauchy-Schwarz shows $(x',x_i')\cdot(e,e_i)\geq0.$
Otherwise, $(e,e_i)\cdot-K\geq0$.  Then $\sum_ie_i^2+e_i\leq e^2+3e$ implies $e\geq e_i$, so Lemma \ref{1} shows $(x',x_i')\cdot(e,e_i)\geq0.$ This completes the proof.
\end{proof}
\begin{remark} By scaling, Lemma \ref{4} extends to $(x,x_i)$ that satisfy a) and b) of Definition \ref{C} with $x,x_i \in \Q$.\end{remark}

\subsection{A key lemma}
We now use the combinatorial machinery from the previous section, together with a reduction to the ball packing problem, to prove the key lemma needed to complete the proof of Theorem~\ref{thm:13/2}, see part (iii) of Lemma~\ref{lem:keylemma} below.  

To reduce to a ball packing problem, note that proposition 1.4 in Frenkel-M\"uller \cite{FM} states that for rational $a$, 
\[E(1,a)\stackrel{s}\hookrightarrow P(\lambda,c\lambda)
\]
if and only if 
\begin{equation}
\label{eqn:frenkelmuller}
E(1,a)\sqcup B(\lambda)\sqcup B(c\lambda) \stackrel{s}\hookrightarrow B((1+c)\lambda),
\end{equation}
where $\sqcup$ denotes disjoint union.  Since, as explained in \cite{H}, one can compute the ECH capacities of the disjoint union in terms of the $\#$ operation, we know that the embedding in \eqref{eqn:frenkelmuller} exists if and only if
\begin{equation}
\label{eqn:keyequation}
N(1,a)\#N(\lambda,\lambda)\#N(c\lambda,c\lambda)\leq N((1+c)\lambda,(1+c)\lambda).
\end{equation}

For the rest of the proof of Theorem~\ref{thm:13/2}, we are looking at intervals for $a$ on which the graph of $d$ is equal to the volume obstruction; we therefore want to show that \eqref{eqn:keyequation} holds with $\lambda=\sqrt\frac{a}{2c}$ (of course for our proof one can specify $c=\frac{13}{2}$, but we state things here in slightly greater generality).  By an argument analogous to the argument used in the proof of Theorem \ref{thm: rigidity} it is sufficient to show \[(\sum_{i}d_{i}^2+d_{i})+e_{1}^2+e_1+e_{2}^2+e_2 \leq d^2+3d\]
implies
\[ \quad (\sum_{i}a_id_i)+c\lambda e_1+\lambda e_2 \leq (1+c)\lambda d\] for all $d,d_i,e_1,e_2 $ non-negative integers. Let $m_1=e_1, m_2=e_2$ and $m_i=d_{i-2}$ for $i\geq 3$ and let $w_i(a)=c\lambda, w_2(a)=\lambda$ and $w_i(a)=a_{i-2}$ for $i\geq 3$. Hence, it is enough to show \[\sum_i m_{i}^2+m_i \leq d^2+3d\]
implies
\[\quad m\cdot w(a) \leq (1+c)\lambda d.\] Let $\mu(d;m)(a)=\frac{m\cdot w(a)}{d}$. The previous condition is equivalent to $\mu(d;m)(a) \leq (1+c)\lambda$. By Lemma \ref{4} it is sufficient to check the case \begin{equation}\label{d1}\sum_i m_{i}^2 = d^2+1,\end{equation}\begin{equation}\label{d3}\sum_i m_i =3d-1.\end{equation} Let $E$ be the set of $(d;m)$ satisfying \eqref{d1} and \eqref{d3} with $d,m_i$ non-negative integers. Define $\epsilon$ by $m=\frac{d}{(1+c)\lambda}w(a)+\epsilon$.  We now have a series of lemmas, culminating in the key Lemma~\ref{lem:keylemma}.

\begin{lemma}\label{21} For $(d;m) \in E$\begin{enumerate}[(i)]
\item$\mu(d;m)(a) \leq (1+c)\lambda \sqrt{1+\frac{1}{d^2}}$ 
\item$\mu(d;m)(a)> (1+c)\lambda$ if and only if $\epsilon \cdot w > 0$
\item $\mu(d;m)(a) >(1+c)\lambda$ implies $\sum_i \epsilon_{i}^2<1$
\item Let $y(a)=a+1-2(1+c)\lambda$. \text{Then} $-\sum_i\epsilon_i=1+\frac{d}{(1+c)\lambda}(y(a)-1/q)$ \text{where} $a=\frac{p}{q}.$
\end{enumerate}\end{lemma}

\begin{proof}(i) follows from $\sum_i w_{i}^2=c^2\lambda^2+\lambda^2+\sum_ia_{i}^2=(1+c)^2\lambda^2$ and Cauchy-Schwarz.
To prove (ii) note
$\epsilon \cdot w=m \cdot w -\frac{d}{(1+c)\lambda}w\cdot w$

 $=d(\frac{m\cdot w}{d}-(1+c)\lambda)$

$=d(\mu(d;m)(a)-(1+c)\lambda)$.
\newline 
To prove (iii) note $\sum_i \epsilon_{i}^2=\epsilon \cdot \epsilon=m \cdot m+\frac{d^2}{(1+c)^2\lambda^2} w\cdot w-\frac{2d}{(1+c)\lambda}m \cdot w$

$=1+d^2(2-\frac{2}{(1+c)\lambda}\frac{m\cdot w}{d})$

$<1$ if $\mu(d;m)(a)>(1+c)\lambda$.
\newline
To prove (iv) note $-\sum_i\epsilon_i=\frac{d}{(1+c)\lambda} \sum_i w_i-\sum_i m_i$

$=\frac{d}{(1+c)\lambda}(a+1-\frac{1}{q}+c\lambda+\lambda)-3d-1$

$=1+\frac{d}{(1+c)\lambda}(a+1-\frac{1}{q}-2(1+c)\lambda)$.\end{proof}

\begin{lemma} Let $(d;m) \in E$ and suppose that $I$ is the maximal nonempty open interval such that $\mu(d;m)(a)>(1+c)\lambda$ for all $a \in I$.Then there exists unique $a_0 \in I$ such that $l(a_0)=l(m)$ where $l(a_0)$ is the length of $w_i(a)$ and $l(m)$ is the number of nonzero terms in $m$. Furthermore, $l(a) \geq l(m)$ for all $a \in I$.  \end{lemma}

\begin{proof} We adapt the proof of lemma 2.1.3 in \cite{MS}. For $i \geq 3, w_i(a)$ is piecewise linear and is linear on open intervals that do not contain an element $a'$ with length $l(a') \leq i$. Therefore, if $l(a)>l(m)$ for all $a \in I, \mu(d;m)(a)-\frac{c\lambda m_1+\lambda m_2}{d}$ is linear on I. This is impossible as $c\lambda (1-\frac{m_1}{d})+\lambda (1-\frac{m_2}{d})$ is concave and $I$ is bounded. Thus there exists $a_0 \in I$ with $l(a_0)\leq l(m)$. If $l(a)<l(m)$ then $\sum_{i\leq l(a)}m_{i}^2<d^2+1$ which implies \[m\cdot w \leq||w||\sqrt{\sum_{i\leq l(a)}m_{i}^2}\leq d||w||=(1+c)\lambda d\] which is impossible for $a \in I$. The proof of uniqueness is the same as in \cite[Lemma. 2.1.3]{MS}.\end{proof}

\begin{lemma} \label{lem:p} Let $(d;m)$ be in $E$ with $\mu(d;m)(a)>(1+c)\lambda$ for some $a$. Let $J={k,...,k+s-1}$ be a block of $s\geq 2$ consecutive integers such that $w_i(a)$ is constant for $i \in J$. Then:

\begin{enumerate}[(i)]
\item  One of the following holds:
\[m_k=\cdots =m_{k+s-1} \quad or\]
\[m_k=\cdots =m_{k+s-2}=m_{k+s-1}+1 \quad or\]
\[m_k-1=m_{k+1}=\cdots =m_{k+s-1}.\]
\item There is at most one block of length $s\geq 2$ on which the $m_i$ are not all equal.
\item  If there is a block $J$ of length $s\geq 2$ on which the $m_i$ are not all equal then $\sum_{i \in J}\epsilon_{i}^2\geq \frac{s-1}{s}$.
\end{enumerate}

\end{lemma}

\begin{proof} See the proof of \cite[Lemma. 2.1.7]{MS}.  Here, McDuff and Schlenk are considering the case of an ellipsoid into a ball, but their proof generalizes without change to our situation. \end{proof}

\begin{lemma}\label{10} Let $(d;m) \in E$ be such that $\mu(d;m)>(1+c)\lambda$ for some $a$ with $l(a)=l(m)=M$. Let $w_{k+1},...,w_{k+s}$ be a block but not the first block of $w(a)$ [the first two terms of w(a) are not considered to be part of any block].\begin{enumerate}[(i)]
\item If this block is not the last block, then \[|m_k-(m_{k+1}+\cdots+m_{k+s}+m_{k+s+1})|<\sqrt{s+2}\]
If this block is the last block, then
\[|m_k-(m_{k+1}+\cdots+m_{k+s})|<\sqrt{s+1}\]
\item Always, 
\[m_k-\sum_{i=k+1}^Mm_i<\sqrt{M-k+1}\]\end{enumerate}\end{lemma}

\begin{proof} Similar to the proof of Lemma \ref{lem:p}, see the proof of \cite[Lemma. 2.1.8]{MS} where McDuff and Schlenk's proof generalizes without change to our situation.\end{proof}

\begin{lemma}\label{lem:keylemma}\label{22} Assume that $(d;m) \in E$ and $\mu(d;m)(a)>(1+c)\lambda$ for some $a$ with $l(a)=l(m)$. Assume further that $y(a)>\frac{1}{q}$. Let $v_M=\frac{d}{q(1+c)}\lambda$ and let $L=l(m)$. Then: \begin{enumerate}[(i)]
\item $|\sum_i\epsilon_i|\leq \sqrt{L}$
\item$v_M>\frac{1}{3}$
\item Let $\delta=y(a)-\frac{1}{q}>0$. Then 
\[d\leq \frac{(1+c)\lambda}{\delta}(\sqrt{L}-1)\leq \frac{(1+c)\lambda}{\delta}(\sqrt{q+\floor{a}+2}-1)\] and
$\sqrt{q+\floor{a}+2} \geq 1+\delta v_Mq.$\end{enumerate}\end{lemma}

\begin{proof} (i) follows from $\sum_i\epsilon_{i}^2<1$. (ii) follows from the same argument as \cite[Lemma. 5.1.2]{MS}. From  \cite[Sublemma. 5.1.1]{MS} $q+\floor{a}+2\geq L$ so Lemma \ref{21} implies \[\sqrt{q+\floor{a}+2} \geq \sqrt{L}\geq 1+\frac{d}{(1+c)\lambda}(y(a)-\frac{1}{q})=1+\frac{d}{(1+c)\lambda}\delta=1+qv_M\delta.\] This also shows $d\leq \frac{(1+c)\lambda}{\delta}(\sqrt{q+\floor{a}+2}-1).$\end{proof}

\section{Proof of Therorem~\ref{thm:13/2}  Part III}
With the Lemma~\ref{lem:keylemma} now shown, we can complete the proof of Theorem~\ref{thm:13/2}.  We explain the computation on various intervals seperately.

\subsection{$ [\frac{1300}{81},\frac{841}{52}]$}

We now wish to prove that $d(a,6.5)=\sqrt{\frac{a}{13}}$ for $a \in [\frac{1300}{81},\frac{841}{52}]$. Previously, we have proved \[d(\frac{1300}{81},6.5)=\frac{10}{9} \quad and \quad d(\frac{841}{52},6.5)=\frac{29}{26}.\] If $d(a,6.5)$ is not $\sqrt{\frac{a}{13}}$ on the interval $[\frac{1300}{81},\frac{841}{52}]$, then there exists $(d;m) \in E$ such that $\mu(d;m)(a)>7.5\lambda$ for some $a \in [\frac{1300}{81},\frac{841}{52}]$. So, Lemma  \ref{22} shows that there exists $a_0$ in $[\frac{1300}{81},\frac{841}{52}]$ with $\mu(d;m)(a_0)>7.5\lambda$ and $l(a_0)=l(m)$. Let $a_0=\frac{p}{q}=16+\frac{p'}{q}$. As $16<a_0<16+\frac{1}{5}$ we know $q\geq5$. For $a_0 \in [\frac{1300}{81},\frac{841}{52}], q\geq5$ we know \[\delta \geq \frac{1300}{81}+1-15\sqrt{\frac{1300}{81*13}}-\frac{1}{q}\geq \frac{31}{81}-\frac{1}{q}.\] Thus, Lemma \ref{22} shows \[\sqrt{q+18}\geq 1+(\frac{31}{81}-\frac{1}{q})\frac{1}{3}q.\] Hence, $q\leq67$. 

We also note that for $\frac{1300}{81}<a_0<\frac{841}{52}, q\geq5$ we have $\lambda\leq\sqrt{\frac{841}{52*13}}=\frac{29}{26}$ and $\delta\geq\frac{31}{81}-\frac{1}{q}\geq\frac{74}{405}$. Thus, Lemma  \ref{22} shows $d\leq \frac{7.5\cdot\frac{29}{26}}{\frac{74}{405}}(\sqrt{85}-1)<377$. Using our code (see Appendix A.2) we can reduce the possibilities for $(d;m)$ to 38 candidates. We can then use Lemma \ref{10} to reduce these 38 cases to 11 possible candidates which can easily verified to not be obstructive by simple calculations.

\subsection{$ [\frac{15028}{841},\frac{961}{52}]$}

We now will show $d(a,6.5)=\sqrt{\frac{a}{13}}$ for $a \in [\frac{15028}{841},\frac{961}{52}]$. Previously, we have proved \[d(\frac{15028}{841},6.5)=\frac{34}{29} \quad and \quad d(\frac{961}{52},6.5)=\frac{31}{26}.\] If $d(a,6.5)$ is not $\sqrt{\frac{a}{13}}$ on the interval $[\frac{15028}{841},\frac{961}{52}]$, then there exists $(d;m) \in E$ such that $\mu(d;m)(a)>7.5\lambda$ for some $a \in [\frac{15028}{841},\frac{961}{52}]$. Then Lemma  \ref{22} shows that there exists $a_0 \in [\frac{15028}{841},\frac{961}{52}]$ with $\mu(d,m)(a_0)>7.5\lambda$ and $l(a_0)=l(m)$. Let $a_0=\frac{p}{q}$ with $gcd(p,q)=1$. For $a_0 \in [\frac{15028}{841},\frac{961}{52}]$ we know \[\delta\geq \frac{15028}{841}+1-15\sqrt{\frac{15028}{841*13}}-\frac{1}{q}= \frac{1079}{841}-\frac{1}{q}.\] Thus, Lemma \ref{22} shows $\sqrt{q+19}\geq1+(\frac{1079}{841}-\frac{1}{q})\frac{q}{3}$. Hence, $q\leq11$. We can then verify these cases directly using our code (see Appendix A.2) to check these cases and we find no obstructions. 

\subsection{$ [\frac{18772}{961},\frac{1089}{52}]$}
We will now show $d(a,6.5)=\sqrt{\frac{a}{13}}$ for $a \in [\frac{18772}{961},\frac{1089}{52}]$. Previously, we have proved \[d(\frac{18772}{961},6.5)=\frac{38}{31} \quad and \quad d(\frac{1089}{52},6.5)=\frac{33}{26}.\] If $d(a,6.5)$ is not $\sqrt{\frac{a}{13}}$ on the interval $[\frac{18772}{961},\frac{1089}{52}]$, then there exists $(d;m) \in E$ such that $\mu(d;m)(a)>7.5\lambda$ for some $a \in [\frac{18772}{961},\frac{1089}{52}]$. Then Lemma \ref{22} shows that there exists $a_0 \in [\frac{18772}{961},\frac{1089}{52}]$ with $\mu(d,m)(a_0)>7.5\lambda$ and $l(a_0)=l(m)$. Let $a_0=\frac{p}{q}$ with $gcd(p,q)=1$. For $a_0 \in [\frac{18772}{961},\frac{1089}{52}]$ we know \[\delta\geq\frac{18772}{961}+1-15\sqrt{\frac{18772}{961*13}}-\frac{1}{q}=
\frac{2063}{961}-\frac{1}{q}.\] Thus, Lemma \ref{22} shows $\sqrt{q+21}\geq1+(\frac{2063}{961}-\frac{1}{q})\frac{q}{3}$. Hence, $ q\leq6$. We can then verify these cases directly using our code (see Appendix A.2) to check these cases and we find no obstructions.

\subsection{$[\frac{2548}{121},27]$}
For $a \in [\frac{2548}{121},27], \sqrt{q+29}\geq\sqrt{q+\floor{a}+2} \text{ and } \delta\geq21-15\sqrt{\frac{21}{13}}.$ Hence, Lemma \ref{22} implies 
\[\sqrt{q+29}\geq1+(21-15\sqrt{\frac{21}{13}})\frac{q}{3}\] which implies $q<8$. We can then verify these cases directly using our code (see Appendix A.2) to check these cases and we find no obstructions.

\subsection{$[27,\infty)$}
We will apply Remark \ref{sharpbound}.  As
\[\sqrt{27}\geq\frac{7.5}{\sqrt{13}}(2+\frac{7.5}{d}) \text{ for } d\geq18\]
Remark \ref{sharpbound} implies we only need to verify $N_k(1,a^2)\leq \frac{a}{13}M_k(1,6.5)$ for all $k\leq\frac{18^2}{26}+\frac{7.5\cdot18}{13}+\frac{6.5^2-13+1}{26}<25.$ For $a^2\geq27, k\leq25, N_k(1,a^2)=k\leq \sqrt{\frac{27}{13}}M_k(1,6.5)\leq\frac{a}{\sqrt{13}}M_k(1,6.5).$ This completes the proof $d(a,b)=\sqrt{\frac{a}{13}} \text{ for } a \in [27,\infty).$

\section{Conjectures}

We now present some conjectures concerning exactly when an ellipsoid embeds into a polydisc.

\subsection{Extensions of Theorem 1.1}
To consider an interesting refinement of Theorem 1.1, define $V(b)=\inf\{A: d(a,b)=\sqrt{\frac{a}{2b}} \text{ for } a\geq A\}$. Theorem 1.1 implies $V(b)\leq\frac{9}{2}(b+2+\frac{1}{b})$.
\begin{proposition} For $b\geq1$
\[V(b)\geq2b\left(\frac{2\floor{b}+2\ceil{\sqrt{2b}+\{b\}}-1}{b+\floor{b}+\ceil{\sqrt{2b}+\{b\}}-1}\right)^2.\]
\end{proposition}
\begin{proof}
\[d(2\floor{b}+2\ceil{\sqrt{2b}+\{b\}}-1,b)\geq\frac{N_{2\floor{b}+2\ceil{\sqrt{2b}+\{b\}}-1}(1,2\floor{b}+2\ceil{\sqrt{2b}+\{b\}}-1)}{M_{2\floor{b}+2\ceil{\sqrt{2b}+\{b\}}-1}(1,b)}\]
\[=\frac{2\floor{b}+2\ceil{\sqrt{2b}+\{b\}}-1}{b+\floor{b}+\ceil{\sqrt{2b}+\{b\}}-1}\]
\[>\sqrt{\frac{2\floor{b}+2\ceil{\sqrt{2b}+\{b\}}-1}{2b}}\]
This implies 
\[V(b)\geq2b\left(\frac{2\floor{b}+2\ceil{\sqrt{2b}+\{b\}}-1}{b+\floor{b}+\ceil{\sqrt{2b}+\{b\}}-1}\right)^2.\]
\end{proof}
Experimental evidence seems to suggest that for $b>1$ this bound is sharp.
\begin{conjecture}
\label{conj1}
For $b>1$
\[V(b)=2b\left(\frac{2\floor{b}+2\ceil{\sqrt{2b}+\{b\}}-1}{b+\floor{b}+\ceil{\sqrt{2b}+\{b\}}-1}\right)^2.\]
\end{conjecture}

\subsection{Generalizations of Theorem 1.2}
The methods used to compute the graph of $d(a,6.5)$ should extend for the most part to any b. In light of those techniques, experimental evidence, and a conjecture regarding $d(a,b)$ for b an integer by David Frenkel and Felix Schlenk relayed to us by Daniel Cristofaro-Gardiner, we offer a conjecture regarding the graph of $d(a,b)$ for $b\geq6$, see Figure $2$.

\begin{conjecture}
\label{conj2}
For $b\geq6,  d(a,b)=\sqrt{\frac{a}{2b}}$ except on the following intervals
\[d(a,b)=1 \text{ for } a\in [1,b+\floor{b}]\]
For $ k \in Z, 0\leq k<\sqrt{2b}+\{b\}$
\[d(a,b)=\frac{a}{b+\floor{b}+k} \text{ for } a\in [\alpha_k,2(\floor{b}+k)+1]\]
\[d(a,b)=\frac{2(\floor{b}+k)+1}{b+\floor{b}+k} \text{ for } a\in [2(\floor{b}+k)+1,\beta_k]\]
where $\alpha_0=b+\floor{b}, \alpha_1=\beta_0=\frac{(b+\floor{b}+1)(2\floor{b}+1)}{b+\floor{b}}, \alpha_k=\frac{(b+\floor{b}+k)^2}{2b} \text{ for } k\geq2, \beta_k=2b\left(\frac{2(\floor{b}+k)+1)}{b+\floor{b}+k}\right)^2 \text{ for } k\geq1.$\\
For integers m if $b \in [m-\frac{m}{(m+1)^2},m+\frac{1}{2+m}]$ let $b=m+\epsilon$
\[d(a,b)=\frac{ma+1}{2m^2+(2+\epsilon)m+\epsilon} \text{ for } a\in[\alpha*,2m+4]\]
\[d(a,b)=\frac{m(2m+4)+1}{2m^2+(2+\epsilon)m+\epsilon} \text{ for } a\in[2m+4,\beta*]\]
where 
$\alpha^*=\frac{1}{2(2m^3+2m^2\epsilon)}(8m^3+4m^2+8m^2\epsilon+4m^3\epsilon+\epsilon^2+2m\epsilon^2+b^2\epsilon^2-(1+m)(2m+\epsilon)\sqrt{-4m^2+8m^3+4m^4-4m\epsilon+8m^2\epsilon+4m^3\epsilon+\epsilon^2+2m\epsilon^2+m^2\epsilon^2})$
and $\beta^*=\frac{2(\epsilon+m+8m\epsilon+8m^2+20m^2\epsilon+16m^3\epsilon+16m^4+4m^4\epsilon+4m^5)}{(1+m)^2(2m+epsilon)^2}.$
\end{conjecture}

\begin{figure}
\includegraphics[height = 75 mm, width = 120 mm]{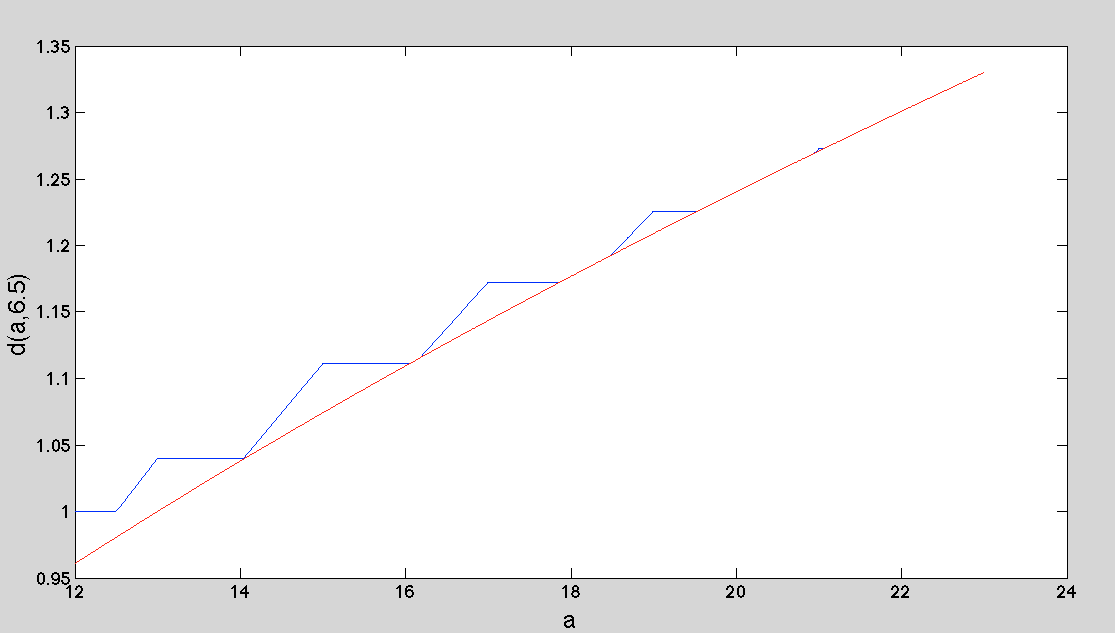}
\caption{Approximate plot of the graph of $d(a,b)$}.
\end{figure}

We note that Conjecture \ref{conj2} implies Conjecture \ref{conj1} for $b\geq6$. Furthermore, we prove that the conjecture is a lower bound for $d(a,b)$.
\begin{proposition}
 For $b\geq6,  d(a,b)\geq\sqrt{\frac{a}{2b}}$ and
\[d(a,b)\geq1 \text{ for } a\in [1,b+\floor{b}]\]
For $ k \in Z, 0\leq k<\sqrt{2b}+\{b\}$
\[d(a,b)\geq\frac{a}{b+\floor{b}+k} \text{ for } a\in [\alpha_k,2(\floor{b}+k)+1]\]
\[d(a,b)\geq\frac{2(\floor{b}+k)+1}{b+\floor{b}+k} \text{ for } a\in [2(\floor{b}+k)+1,\beta_k]\]
where $\alpha_0=b+\floor{b}, \alpha_1=\beta_0=\frac{(b+\floor{b}+1)(2\floor{b}+1)}{b+\floor{b}}, \alpha_k=\frac{(b+\floor{b}+k)^2}{2b} \text{ for } k\geq2,\beta_k=2b\left(\frac{2(\floor{b}+k)+1)}{b+\floor{b}+k}\right)^2 \text{ for } k\geq1.$\\
For integers m if $b \in [m-\frac{m}{(m+1)^2},m+\frac{1}{2+m}]$ let $b=m+\epsilon$
\[d(a,b)\geq\frac{ma+1}{2m^2+(2+\epsilon)m+\epsilon} \text{ for } a\in[\alpha*,2m+4]\]
\[d(a,b)\geq\frac{m(2m+4)+1}{2m^2+(2+\epsilon)m+\epsilon} \text{ for } a\in[2m+4,\beta*]\]
where 
$\alpha^*=\frac{1}{2(2m^3+2m^2\epsilon)}(8m^3+4m^2+8m^2\epsilon+4m^3\epsilon+\epsilon^2+2m\epsilon^2+b^2\epsilon^2-(1+m)(2m+\epsilon)\sqrt{-4m^2+8m^3+4m^4-4m\epsilon+8m^2\epsilon+4m^3\epsilon+\epsilon^2+2m\epsilon^2+m^2\epsilon^2})$
and $\beta^*=\frac{2(\epsilon+m+8m\epsilon+8m^2+20m^2\epsilon+16m^3\epsilon+16m^4+4m^4\epsilon+4m^5)}{(1+m)^2(2m+\epsilon)^2}.$
\end{proposition}
\begin{proof}
We know that $d(a,b)\geq\sqrt{\frac{a}{2b}}$ because symplectic embeddings are volume preserving. We also have
\[d(a,b)\geq\frac{N_1(1,a)}{M_1(a,b)}=\frac{1}{1}=1.\]
Additionally, for $k \in \Z, 0\leq k<\sqrt{2b}+\{b\}, a\in [2(\floor{b}+k),2(\floor{b}+k)+1]$
\[d(a,b)\geq\frac{N_{2(\floor{b}+k)+1}(1,a)}{M_{2(\floor{b}+k)+1}(1,b)}=\frac{a}{b+\floor{b}+k}\]
\[\geq1 \text{ for } a\in[b+\floor{b},2\floor{b}+1], k=0\]
\[\geq\frac{2\floor{b}+1}{b+\floor{b}} \text { for } a\in[\frac{(b+\floor{b}+1)(2\floor{b}+1)}{b+\floor{b}},2\floor{b}+3], k=1\] 
\[\geq\sqrt{\frac{a}{2b}} \text{ for } a\in[\alpha_k,2(\floor{b}+k)+1], k\geq2.\]
We also have for $a \in [2(\floor{b}+k)+1,\infty)$
\[d(a,b)\geq\frac{N_{2(\floor{b}+k)+1}(1,a)}{M_{2(\floor{b}+k)+1}(1,b)}=\frac{2(\floor{b}+k)+1}{b+\floor{b}+k}\]

\[\geq\sqrt{\frac{a}{2b}} \text { for } a\in[2(\floor{b}+k)+1,\beta_k].\]
Furthermore, if $b\in[m-\frac{m}{(m+1)^2},m+\frac{1}{2+m}]$ for some $m \in \Z \text{ and } a\in[2m+4-\frac{1}{m},2m+4]$
\[d(a,b)\geq\frac{N_{(m+1)^3}(1,a)}{M_{(m+1)^3}(1,b)}=\frac{ma+1}{2m^2+(2+\epsilon)m+\epsilon}\]
\[\geq\sqrt{\frac{a}{2b}} \text{ for } a\in[\alpha^*,2m+4].\]
We also have for $a\in[2m+4m\beta^*]$
\[d(a,b)\geq\frac{N_{(m+1)^3}(1,a)}{M_{(m+1)^3}(1,b)}=\frac{m(2m+4)+1}{2m^2+(2+\epsilon)m+\epsilon}\]
\[\geq\sqrt{\frac{a}{2b}} \text{ for } a\in[2m+4,\beta^*].\]
This completes the proof.
\end{proof}

\appendix
\section{Appendix}

\subsection{Code that checks through terms of $N$ and $M$}
\label{app:embedcode}

The following is Matlab code that allows us to check whether $N(1,b)\leq M(c,d)$ up through $N(1,a)\leq x$ (note that for the function {\em Membed}, $d\leq c$):

\begin{verbatim}
function m=embed(b,c,d,x)
l=length(Nembed(x,b));
y=zeros(1,l); w=Nembed(x,b); t=Membed(l,c,d);
for i=1:l
    if w(i)<=(t(i)+10^-10)    
        y(i)=1;        
    else    
        m=0;        
    break
    end
end
m=min(y);

function y=Nembed(x,b);
y=zeros(1,x+2);
for i=1:x+2
    y(i)=floor((x+b-(i-1))/b);
end
M=sum(y);
z=zeros(1,(x+1)^2);
for i=1:x+1
    for j=1:x+1    
        z(i+(x+1)*(j-1))=i-1+(j-1)*b;
    end
end
l=sort(z);
y=zeros(1,M);
for i=1:M
    y(i)=l(i);    
end

function q=Membed(N,c,d)
q=zeros(1,N);
for k=0:N-1
    w=zeros(1,ceil(sqrt(k+1)-1));    
    for i=0:(ceil(sqrt(k+1)-1))    
        w(i+1)=c*i+d*(ceil((k+1)/(i+1))-1);        
    end    
    q(k+1)=min(w);    
end
\end{verbatim}

\newpage
\subsection{Mathematica Code}
\text{W[a\_] outputs the weight sequence for a.}\\
\noindent\({W[\text{a$\_$}]\text{:=}\text{Module}[\{\text{aa}=a,M,i=2,L,u,v\}},\\
 \indent {M=\text{ContinuedFraction}[\text{aa}];}\\
\indent {L=\text{Table}[1,\{j,M[[1]]\}];}\\
\indent {\{u,v\}=\{1,\text{aa}-\text{Floor}[\text{aa}]\};}\\
\indent{\text{While}[i\leq \text{Length}[M],L=\text{Join}[L,\text{Table}[v,\{j,M[[i]]\}]];}\\
\indent \indent{\{u,v\}=\{v,u-M[[i]] v\};}\\
\indent \indent{i\text{++}];}\\
\indent{\text{Return}[L]]}\\
$
{\text{GenW[a\_] outputs} $(c\lambda,\lambda)$ \text{ joined with the weight sequence of a.}}\\
{\text{GenW}[\text{a$\_$},\text{c$\_$}] \text{:=}\text{Join}[\{c*\text{Sqrt}[a/(2*c)],\text{Sqrt}[a/(2*c)]\},W[a]]}\\
{\text{P2}[\text{k$\_$}]\text{:=}\text{Module}[\{\text{kk}=k,\text{PP},\text{T0},i\},\text{T0}=\text{Table}[0,\{u,1,k\}];}\\
\indent{\text{T0p}=\text{ReplacePart}[\text{T0},1,1];}\\
\indent{\text{T11}=\text{Table}[1,\{u,1,k\}];}\\
\indent{\text{T1m}=\text{ReplacePart}[\text{T11},0,-1];}\\
\indent{\text{PP}=\{\text{T0},\text{T0p},\text{T11},\text{T1m}\};}\\
\indent{\text{Return}[\text{PP}]]}\\
{\text{Difference}[\text{M$\_$}]\text{:=}\text{Module}[\{V=M,\text{vN},\text{V1},l,L=\{\},D,\text{PP},i,j,N\},l=\text{Length}[V];}\\
\indent{\text{If}[l==1,L=\text{P2}[V[[1]]]];}\\
\indent{\text{If}[l>1,\text{vN}=V[[-1]];}\\
\indent \indent{\text{V1}=\text{Delete}[V,-1];}\\
\indent \indent{D=\text{Difference}[\text{V1}];}\\
\indent \indent{\text{PP}=\text{P2}[\text{vN}];}\\
\indent \indent{i=1;}\\
\indent \indent${\text{While}[i\leq\text{Length}[D],j=1;}\\
\indent \indent\indent{\text{While}[j\leq\text{Length}[\text{PP}],}$\\
\indent \indent\indent\indent{N=\text{Join}[ D[[i]],\text{PP}[[j]]];}\\
\indent\indent\indent\indent{L=\text{Append}[L,N];}\\
\indent\indent\indent\indent{j\text{++}];}\\
\indent\indent\indent{i\text{++}]}\\
\indent{];}\\
\indent{\text{Return}[L]]}\\
{\text{Sol}[\text{a$\_$},\text{d$\_$},\text{c$\_$}]\text{:=}\text{Module}[\{\text{aa}=a,\text{dd}=d,\text{cc}=c, A,M,F,D,i,V,L=\{\}\},}\\
\indent{A=\text{ContinuedFraction}[\text{aa}];}\\
\indent{M=\text{Join}[\{1,1\},A];}\\
\indent{F=\text{Floor}[\text{dd}/((1+\text{cc})*\text{Sqrt}[\text{aa}/(2*\text{cc})]) \text{GenW}[\text{aa},\text{cc}]];}\\
\indent{D=\text{Difference}[M];}\\
\indent{i=1;}\\
\indent${\text{While}[i\leq \text{Length}[D],V=\text{Sort}[F+D[[i]],\text{Greater}];}$\\
\indent\indent{\text{SV}=\text{Sum}[V[[j]],\{j,1,\text{Length}[V]\}];}\\
\indent\indent${\text{If}[\{\text{SV},V.V\}==\{3*\text{dd}-1,\text{dd}{}^{\wedge}2+1\}\&\&V[[-1]]>0\&\&}\\
\indent\indent\indent{\text{GenW}[\text{aa},\text{cc}].V/\text{dd}\geq ((1+\text{cc})*\text{Sqrt}[\text{aa}/(2*\text{cc})]),L=\text{Append}[L,V]];}\\
\indent\indent{i\text{++}];}\\
\indent{\text{Return}[\{\text{dd},\text{Union}[L]\}]]}\\
{\text{SolLess}[\text{a$\_$},\text{D$\_$},\text{c$\_$}]\text{:=}}
{\text{Module}[\{\text{aa}=a,\text{DD}=D,\text{cc}=c, d=1,\text{Ld},L=\{\}\},}\\
\indent{\text{While}[d\leq D,\text{Ld}=\text{Sol}[\text{aa},d,\text{cc}];}\\
\indent\indent{\text{If}[\text{Length}[\text{Ld}[[2]]]>0,L=\text{Append}[L,\text{Sol}[\text{aa},d,\text{cc}]]];}\\
\indent\indent{d\text{++}];}\\
\indent{\text{Return}[L]]}\\
{\text{Sol2}[\text{a$\_$},\text{d$\_$},\text{c$\_$}]\text{:=}\text{Module}[\{\text{aa}=a,\text{dd}=d,\text{cc}=c, A,M,F,D,i,V,L=\{\}\},}\\
\indent{A=\text{ContinuedFraction}[\text{aa}];}\\
\indent{M=\text{Join}[\{1,1\},A];}\\
\indent{F=\text{Floor}[\text{dd}/((1+\text{cc})*\text{Sqrt}[\text{aa}/(2*\text{cc})]) \text{GenW}[\text{aa},\text{cc}]];}\\
\indent{D=\text{Difference}[M];}\\
\indent{i=1;}\\
\indent{\text{While}[i\leq \text{Length}[D],V=\text{Sort}[F+D[[i]],\text{Greater}];}\\
\indent\indent{\text{SV}=\text{Sum}[V[[j]],\{j,1,\text{Length}[V]\}];}\\
\indent\indent{\text{If}[\{\text{SV},V.V\}==\{3*\text{dd}-1,\text{dd}{}^{\wedge}2+1\}\&\&}\\
\indent\indent\indent{\text{GenW}[\text{aa},\text{cc}].V/\text{dd}\geq ((1+\text{cc})*\text{Sqrt}[\text{aa}/(2*\text{cc})]+10{}^{\wedge}-10),}\\
\indent\indent{L=\text{Append}[L,V]];}\\
\indent{i\text{++}];}\\
\indent{\text{Return}[\{\text{dd},\text{Union}[L]\}]]}\\
{\text{SolLess2}[\text{a$\_$},\text{D$\_$},\text{c$\_$}]\text{:=}}
{\text{Module}[\{\text{aa}=a,\text{DD}=D,\text{cc}=c, d=1,\text{Ld},L=\{\}\},}\\
\indent{\text{While}[d\leq D,\text{Ld}=\text{Sol2}[\text{aa},d,\text{cc}];}\\
\indent\indent{\text{If}[\text{Length}[\text{Ld}[[2]]]>0,L=\text{Append}[L,\text{Sol2}[\text{aa},d,\text{cc}]]];}\\
\indent\indent{d\text{++}];}\\
\indent{\text{Return}[L]]}\\
{\text{Solutions}[\text{a$\_$},\text{b$\_$}]\text{:=}\text{Solutions}[a,b,\text{Min}[a,\text{Floor}[\text{Sqrt}[b]]]]}\\
{\text{Solutions}[\text{a$\_$},\text{b$\_$},\text{c$\_$}]\text{:=}}
{\text{Module}[\{A=a,B=b,C=c,i,m,K,j,V,L=\{\}\},\text{If}[A{}^{\wedge}2<B,L=\{\}];}\\
\indent{\text{If}[A{}^{\wedge}2==B,}\\
\indent\text{If}[A>C,L=\{\},L=\{\{A\}\}]];\\
\indent{\text{If}[A{}^{\wedge}2>B,i=1;}\\
\indent\indent{m=\text{Min}[\text{Floor}[\text{Sqrt}[B]],C];}\\
\indent\indent{\text{While}[i\leq m,K=\text{Solutions}[A-i,B-i{}^{\wedge}2,i];}\\
\indent\indent\indent{j=1;}\\
\indent\indent\indent{\text{While}[j\leq \text{Length}[K],V=\text{Prepend}[K[[j]],i];}\\
\indent\indent\indent\indent{L=\text{Append}[L,V];}\\
\indent\indent\indent{j\text{++}];}\\
\indent\indent{i\text{++}]];}\\
\indent{\text{Return}[\text{Union}[L]]]}\\
{T \text{:=}\text{Table}[0,\{u,1,18\}]}\\
{\text{T1} \text{:=}\text{ReplacePart}[T,1,1]}\\
{\text{T2} \text{:=} \text{ReplacePart}[T,1,2]}\\
{\text{T3} \text{:=}\text{ReplacePart}[T,1,3]}\\
{\text{T4}\text{:=} \text{ReplacePart}[T,-1,18]}\\
{P=\{T,\text{T3},\text{T4},\text{T1},\text{T1}+\text{T3},\text{T1}+\text{T4},\text{T2},\text{T2}+\text{T3},\text{T2}+\text{T4},\text{T1}+\text{T2},}\\
{\text{T1}+\text{T2}+\text{T3},\text{T1}+\text{T2}+\text{T4}\}}\\
{\text{a1}[\text{d$\_$}]\text{:=}\text{Floor}[d*13/15]}\\
{\text{a2}[\text{d$\_$}] \text{:=} \text{Floor}[d*2/15]}\\
{\text{a3}[\text{d$\_$}] \text{:=} \text{Round}[d*2/15/\text{Sqrt}[16.1/13]]}\\
{Q[\text{d$\_$}]\text{:=}\text{Table}[\text{a3}[d],\{u,1,18\}]}\\
{\text{Q1}[\text{d$\_$}] \text{:=}\text{ReplacePart}[Q[d],\text{a1}[d],1]}\\
{\text{Q2}[\text{d$\_$}] \text{:=} \text{ReplacePart}[\text{Q1}[d],\text{a2}[d],2]}\\
{\text{sum}[\text{L$\_$}]\text{:=}\text{Sum}[L[[j]],\{j,1,\text{Length}[L]\}]}\\
{\text{Squaresum}[\text{L$\_$}] \text{:=}\text{Sum}[(L[[j]]){}^{\wedge}2,\{j,1,\text{Length}[L]\}]}\\
{\text{Solutions2}[\text{d$\_$}] \text{:=}\text{Module}[\{\text{dd}=d,i,K\},}\\
\indent{L=\{\};}\\
\indent{\text{For}[i =1,i<13,i\text{++},}\\
\indent\indent{K=\text{SolutionsAlt}[3*\text{dd}-1-\text{sum}[\text{Q2}[\text{dd}]+P[[i]]],}\\
\indent\indent\indent{\text{dd}{}^{\wedge}2+1-\text{Squaresum}[\text{Q2}[\text{dd}]+P[[i]]],\text{Q2}[\text{dd}]+P[[i]]];}\\
\indent\indent{L=\text{Join}[L,K];}\\
\indent\indent{]}\\
\indent{]}\\
{\text{SolVerify}[\text{d$\_$}] \text{:=}\text{Module}[\{k=\text{Length}[\text{Sol}[d]],T=\text{Sol}[d],P=\{\}\},}\\
\indent{\text{If}[k\text{==}0,P=P,\text{If}[\text{Length}[T[[1]]]>1,}\\
\indent\indent{ \text{For}[i=1,i\leq k,i\text{++},}\\
\indent\indent\indent{\text{If}[T[[i]][[2]]\geq T[[i]][[3]] }\\
\indent\indent\indent\indent{\&\&\text{If}[\text{Length}[T[[i]]]\geq 19,T[[i]][[18]]\geq T[[i]][[19]],\text{True}]}\\
\indent\indent\indent\indent{\&\& \text{If}[\text{Length}[T[[i]]]\text{$>$=}20,}\\
\indent\indent\indent\indent{\text{Length}[T[[i]]]\geq 22 \&\&T[[i]][[19]]\text{$<$=}((T[[i]][[20]])+1),\text{True}]}\\
\indent\indent\indent\indent{\&\& \text{If}[\text{Length}[T[[i]]]\text{$>$=}21,T[[i]][[20]]\text{==}((T[[i]][[21]])),\text{True}]}\\
\indent\indent\indent\indent{\&\& \text{If}[\text{Length}[T[[i]]]\text{$>$=}22,T[[i]][[20]]\text{==}((T[[i]][[22]])),\text{True}]}\\
\indent\indent\indent\indent{\&\& \text{If}[\text{Length}[T[[i]]]\text{$>$=}23,T[[i]][[22]]\text{$<$=}((T[[i]][[23]])+1),\text{True}]}\\
\indent\indent\indent\indent{\&\& \text{Test4}[T[[i]]],}\\
\indent\indent\indent\indent{P=\text{Append}[P,T[[i]]]]],}\\
\indent\indent\indent{\text{If}[T[[2]]\geq T[[3]] }\\
\indent\indent\indent\indent{\&\& \text{If}[\text{Length}[T]\geq 19, T[[18]]\geq T[[19]] ,\text{True}]}\\
\indent\indent\indent\indent{\&\&\text{If}[\text{Length}[T]\geq 20, \text{Length}[T]\geq  22 \&\&T[[19]]\leq (T[[20]]+1),\text{True}]}\\
\indent\indent\indent\indent{\&\&\text{If}[\text{Length}[T]\text{$>$=}21,T[[20]]\text{==}((T[[21]])),\text{True}]}\\
\indent\indent\indent\indent{\&\& \text{If}[\text{Length}[T]\text{$>$=}22,T[[20]]\text{==}((T[[22]])),\text{True}]}\\
\indent\indent\indent\indent{\&\& \text{If}[\text{Length}[T]\text{$>$=}23,T[[22]]\text{$<$=}((T[[23]])+1),\text{True}]}\\
\indent\indent\indent\indent{\&\& \text{Test4}[T],}\\
\indent\indent\indent\indent{P=\text{Append}[P,T]]]];}\\
\indent\indent\indent\indent{\text{Return}[P]]}\\
{\text{AllSol}[\text{D$\_$}] \text{:=}\text{Module}[\{\text{DD}=D,P=\{\}\},}\\
\indent{\text{For}[j=2,j\leq \text{DD},j\text{++},}\\
\indent\indent{\text{Print}[\text{Join}[\{j\},\text{SolVerify}[j]]]]]}\\
{\text{Sol}[\text{d$\_$}]\text{:=}\text{Module}[\{e=d\},}\\
\indent{\text{Solutions2}[e];}\\
\indent{L]}\\
{\text{alpha}[\text{bb$\_$},\text{kk$\_$}]\text{:=}\text{Module}[\{b=\text{bb},k=\text{kk}\},}\\
\indent{\text{If}[k==0,2*b-1/2,}\\
\indent\indent{\left.\left.\text{If}\left[k==1, \frac{2 b (1+4 b)}{-1+4 b},\frac{(-1+4 b+2 k)^2}{8 b}\right]\right]\right]}\\
{\text{beta}[\text{bb$\_$},\text{kk$\_$}]\text{:=}\text{Module}[\{b=\text{bb},k=\text{kk}\},}\\
\indent{\left.\text{If}\left[k==0,\frac{2 b (1+4 b)}{-1+4 b},\frac{32 \left(b^3+2 b^2 k+b k^2\right)}{(-1+4 b+2 k)^2}\right]\right]}\\
{\text{SolutionsAlt}[\text{a$\_$},\text{b$\_$},\text{v$\_$}]\text{:=}}\\
\indent{\text{If}[a\geq  0\&\&b\text{$>$=}0,\text{SolutionsAlt}[a,b,\text{Min}[a,\text{Floor}[\text{Sqrt}[\text{Max}[0,b]]]],v],\{\}]}\\
{\text{SolutionsAlt}[\text{a$\_$},\text{b$\_$},\text{c$\_$},\text{v$\_$}]\text{:=}}\\
\indent{\text{Module}[\{A=a,B=\text{Max}[0,b],C=c,i,m,K,j,V,\text{vv}=v, L=\{\}\},}\\
\indent\indent{\text{If}[A{}^{\wedge}2<B,L=\{\}];}\\
\indent\indent{\text{If}[A{}^{\wedge}2==B,\text{If}[A>C \| A==0,L=\{\},L=\{\text{Join}[\text{vv},\{A\}]\}]];}\\
\indent\indent{\text{If}[A{}^{\wedge}2>B,i=1;}\\
\indent\indent\indent{m=\text{Min}[\text{Floor}[\text{Sqrt}[B]],C];}\\
\indent\indent\indent{\text{While}[i\leq m,K=\text{Solutions}[A-i,B-i{}^{\wedge}2,i];}\\
\indent\indent\indent\indent{j=1;}\\
\indent\indent\indent\indent{\text{While}[j\leq \text{Length}[K],V=\text{Prepend}[K[[j]],i];}\\
\indent\indent\indent\indent\indent{\quad V=\text{Join}[\text{vv},V];}\\
\indent\indent\indent\indent{L=\text{Append}[L,V];}\\
\indent\indent\indent\indent{j\text{++}];}\\
\indent\indent\indent{i\text{++}]];}\\
\indent{\text{Return}[L]]}\\
{\text{Test4}[\text{v$\_$}]\text{:=}\text{Module}[\{\text{vv}=v,k=0,T=v[[18]]\},}\\
\indent{\text{While}[\text{Length}[\text{vv}]\geq (19+k) \&\& (\text{vv}[[19+k]]+1)\geq \text{vv}[[19]],}\\
\indent\indent{T=(T-\text{vv}[[19+k]]);k\text{++}];}\\
\indent{T<\text{Sqrt}[k+2]]}\\
{\text{Verify2}\text{:=}\text{For}[\text{qq}=1,\text{qq}\leq 11,\text{qq}\text{++},}\\
\indent{\text{For}[\text{pp}=1,\text{pp}<2*\text{qq},\text{pp}\text{++},}\\
\indent\indent{\text{If}[\text{GCD}[\text{pp},\text{qq}]==1 \&\& 15028/841\text{$<$=}17+\text{pp}/\text{qq}\leq 961/52,}\\
\indent\indent\indent{\text{Print}[\{17+\text{pp}/\text{qq},\text{SolLess}[17+\text{pp}/\text{qq},}\\
\indent\indent\indent\indent{15/2*\text{Sqrt}[(17+\text{pp}/\text{qq})/13]/}\\
\indent\indent\indent\indent\indent{(17+(\text{pp}-1)/\text{qq}+1-15*\text{Sqrt}[(17+\text{pp}/\text{qq})/13])*}\\
\indent\indent\indent\indent(\text{Sqrt}[\text{qq}+20]-1),
{13/2]\}]]]]}\\
{\text{Verify3}\text{:=}\text{For}[\text{qq}=1,\text{qq}\leq 6,\text{qq}\text{++},}\\
\indent{\text{For}[\text{pp}=1,\text{pp}<2*\text{qq},\text{pp}\text{++},}\\
\indent\indent{\text{If}[\text{GCD}[\text{pp},\text{qq}]==1 \&\& 18772/961\text{$<$=}19+\text{pp}/\text{qq}\leq 1089/52,}\\
\indent\indent\indent{\text{Print}[\{19+\text{pp}/\text{qq},\text{SolLess}[19+\text{pp}/\text{qq},}\\
\indent\indent\indent\indent{15/2*\text{Sqrt}[(19+\text{pp}/\text{qq})/13]/}\\
\indent\indent\indent\indent\indent{(19+(\text{pp}-1)/\text{qq}+1-15*\text{Sqrt}[(19+\text{pp}/\text{qq})/13])*}\\
\indent\indent\indent\indent(\text{Sqrt}[\text{qq}+22]-1), {13/2]\}]]]]}\\
{\text{Verify4}\text{:=}\text{For} [\text{qq}=1,\text{qq}<8,\text{qq}\text{++},}\\
\indent{\text{For}[\text{pp}=1,\text{pp}\text{$<$=}6*\text{qq},\text{pp}\text{++},}\\
\indent\indent{\text{If}[\text{GCD}[\text{pp},\text{qq}]==1,}\\
\indent\indent\indent{\text{Print}[\{21+\text{pp}/\text{qq},\text{SolLess}[21+\text{pp}/\text{qq},}\\
\indent\indent\indent\indent{15/2*\text{Sqrt}[(21+\text{pp}/\text{qq})/13]/}\\
\indent\indent\indent\indent\indent{(22+(\text{pp}-1)/\text{qq}-15*\text{Sqrt}[(21+\text{pp}/\text{qq})/13])}\\
\indent\indent\indent\indent{(\text{Sqrt}[\text{qq}+29]-1),13/2]\}]]]]}
\)



\end{document}